\documentclass[11pt,reqno]{amsart}
\usepackage[height=8in,margin=1.125in]{geometry}

\usepackage{amsfonts,amssymb,latexsym,amsmath}
\usepackage{mathtools}

\newtheorem{theorem}{Theorem}[section]
\newtheorem{lemma}[theorem]{Lemma}
\newtheorem{corollary}[theorem]{Corollary}

\theoremstyle{definition}

\newtheorem{example}[theorem]{Example}

\numberwithin{equation}{section}

\DeclareMathOperator{\tr}{trace}
\newcommand{\T}{{\mathbb T}}
\newcommand{\iy}{\infty}
\newcommand{\al}{\alpha}
\newcommand{\sgn}{{\rm sgn}}

\title[Determinant formulas for Toeplitz plus Hankel matrices]{Determinant formulas for finite Toeplitz plus Hankel matrices with rational symbols}
\author{Estelle Basor}
\address{American Institute of Mathematics, Pasadena, CA, USA}
\email{ebasor@aimath.org}
\author{Kent E. Morrison }
\address{American Institute of Mathematics, Pasadena, CA, USA}
\email{morrison@aimath.org}

\begin{document}

\begin{abstract}
In 1975 K. Michael Day produced an exact formula for the determinants of  finite Toeplitz matrices whose symbols are rational. The answer is a sum that involves powers of the roots of the numerator of the symbol and whose coefficients depend on both the roots of the numerator and denominator. In this paper we prove an analogue of Day's formula for determinants of finite Toeplitz plus Hankel matrices with rational symbols. The key to the proof is an exact formula for the finite determinants that involves a Fredholm determinant that can be explicitly computed. We apply the formula to find information about the limiting eigenvalues of the finite Toeplitz plus Hankel matrices. 
\end{abstract}

\maketitle
\section{introduction}
In 1975 K. Michael Day produced an exact formula for the determinants of  finite Toeplitz matrices whose symbols are rational functions \cite{D75a, D75b}. The answer is a sum that involves powers of the roots of the numerator of the symbol and whose coefficients depend on both the roots of the numerator and denominator. 

The proof given in Day's paper used fairly complicated determinant calculations. A sketch of an alternate proof is given in \cite{BS06} that reduces the problem to an expansion using the Cauchy-Binet formula and Vandermonde determinants. 

In 1952, Szeg\H{o} derived the first result for the asymptotics of determinants of finite Toeplitz matrices whose symbols are real-valued and positive. This result was later generalized to include more general complex-valued symbols and so, in some sense, it overlapped with Day's result. However, Szeg\H{o}'s result was asymptotic and Day's was exact. 

The link between the two is via a more recently discovered exact identity for the determinants, namely the Borodin-Okounkov-Case-Geronimo (BOCG) identity. The identity has a factor which is a Fredholm determinant and in the case of rational symbols has a finite expansion. The terms of the expansion can be matched with the summands in Day's result. A proof using this idea was given in \cite{BC} although the notation for the coefficients is expressed in a different way.

Our goal in this paper is not to re-derive Day's result, but rather to extend his result to the case of determinants of finite Toeplitz plus Hankel (T+H) matrices for rational symbols. And, fortunately, in this setting the analogue of the BOCG identity is known. 

Here is an outline of the paper. In the next section we review the finite Toeplitz case and state the BOCG identity. In the next section we state Day's result and sketch a proof of Day's result using the BOCG identity. We also show how the two formulas for the determinants, that is, the one given by Day and the one found in \cite{BC}, are the same. 

Then in Section 4, we state the analogue of the BOCG identity for determinants of Toeplitz plus Hankel matrices. Using the identity we derive the analogue of Day's result for determinants of the finite Toeplitz plus Hankel matrices with rational symbols. This is the main theorem of the paper. 

This is followed by a section that focuses on numerical computations. Verification of the formula is given along with numerical computation of the eigenvalues of the finite matrices. One of the applications and motivations for Day's formula is that it showed that the limiting measures for the eigenvalues of the finite matrices has support on certain arcs. As we shall see in the (T+H) setting the situation is not the same. The main theorem implies that the eigenvalues can only have limiting values on a finite set of arcs, one of which is the same that arises in the Toeplitz case. The actual pictures of the eigenvalues indicate that the T+H case is the same as the Toeplitz case, but we have not been able to prove this.

\section{Toeplitz  Preliminaries}
We begin with a complex-valued function $\phi$ defined on the unit circle $\T$ with Fourier coefficients
$$\phi_{k} = \frac{1}{2\pi} \int_{0}^{2\pi} \phi(e^{i\theta})\,e^{-ik\theta}\,\,d\theta .$$
Hence $$ \phi(e^{i\theta}) = \sum_{-\infty}^{\infty} \phi_{k} \,e^{i\,k\,\theta}.$$
 Consider the matrix 

$$T_n(\phi)\, = \, (\phi_{j-k})_{j,\,k\, = \,0,\, \cdots,\, n-1}.$$

The matrix $T_{n}(\phi)$ is referred to as the finite Toeplitz matrix with symbol $\phi.$

This matrix has the form 

\[ \left[\begin{array}{ccccc}\phi_{0} & \phi_{-1} & \phi_{-2} &  \cdots&  \phi_{-(n-1)} \\
\phi_{1} & \phi_{0} & \phi_{-1} &  \cdots  & \phi_{-(n-2)} \\
\phi_{2} & \phi_{1} & \phi_{0} & \cdots  & \phi_{-(n-3)} \\
\cdots & \cdots & \cdots & \cdots  & \cdots \\
\cdots & \cdots & \cdots&\cdots & \cdots \\
\phi_{n-1} & \phi_{n-2} & \phi_{n-3} & \cdots & \phi_{0}
\end{array}\right]
\]

As mentioned earlier, for a certain class of ``nice'' symbols the asymptotics were described by Szeg\H{o}and then generalized by others to more general symbols. In particular, the generalization to matrix-valued symbols was proved by Widom \cite{Wi76} and the result now is referred to as  the 
Szeg\H{o}-Widom Limit Theorem. For scalar symbols it states that if 
$\phi$  has a sufficiently well-behaved logarithm then the determinant of the Toeplitz
matrix
has the asymptotic behavior
$$  D_{n}(\phi) := \det T_n(\phi)  \, \sim G(\phi)^n\,E(\phi)\ \ \ {\rm as}\ n\to\iy.$$ 

The constant $G(\phi)$ is the geometric mean of $\phi$
$$G(\phi) = e^{(\log \phi)_0}$$
and
$$ E(\phi)  =  \exp\left(\sum_{k=1}^{\iy}k\,(\log\phi)_k\,(\log\phi)_{-k}\right).$$

To clarify what it means for the symbol of have a sufficiently well-behaved logarithm, it is useful to consider a certain Banach algebra (see \cite{BS06}, section 1.10).

Let $B$ be the set of all function $\phi$ such that the Fourier coefficients satisfy
\[ \| \phi\|_{B} :=\sum_{k =-\iy} ^{\iy} |\phi_{k}| + 
 \Big(\sum_{k = -\iy}^{\iy} |k|\cdot|\phi_{k}|^{2}\Big)^{1/2} < \iy. \]
With this norm  and pointwise defined algebraic operations on $\T,$ the set $B$ becomes a Banach algebra of continuous functions.

The Szeg\H{o}-Widom Limit Theorem  holds provided that
$\phi\in B,$ does  not vanish on $\T,$ and has winding number zero. Then both $\phi^{-1}$ and $\log \phi$ (any continuous choice of the logarithm will suffice) are defined and in $B.$

The BOCG identity is a refinement of the Szeg\H{o}-Widom theorem and yields a direct way to prove it. It requires some additional definitions and notation.

We let $\ell^{2}$ denote the space of complex sequences $\{x_k\}_{k=0}^\infty$ with the usual $2$-norm.
With respect to the standard basis in matrix notation we define two operators, the Toeplitz operator, $T(\phi)$ and the Hankel operator $H(\phi)$ by the following:
\begin{eqnarray*}
 T(\phi) & =&  (\phi_{j-k}), \,\,\,\,\,\,\,\,\,\,\,\,\,0 \leq j,k < \infty,\\
 H(\phi) &=&  (\phi_{j+k+1}), \,\,\,\,\,\,\,\,0 \leq j,k < \infty.
 \end{eqnarray*}
 
 It is useful to think of the Toepltiz operator in matrix form, but it is also worth pointing out that it can be defined as 
 \[ T(\phi) : \ell^{2} \rightarrow \ell^{2} \,\,\,\, \quad T(\phi)f = P(\phi f) \] 
 where $P$ is the Riesz projection onto $\ell^{2},$ 
 \[ P : \sum_{k=-\infty}^{\infty}\phi_{k}e^{ik\theta} \mapsto \sum_{k = 0}^{\infty}\phi_{k}e^{ik\theta} ,\] and $f$ is the function identified by its Fourier coefficients $\{f_{n}\}_{n=0}^{\infty}.$ In other words, start with the sequence $\{f_{n}\}_{n = 0}^{\infty}$ in $\ell^{2}$  and let 
 \[ f(\theta)  = \sum_{n=0}^{\infty} f_{n} e^{i n \theta} .\] Multiply $f$ pointwise by the function $\phi$ and then project by chopping off the negative Fourier coefficients of $\phi f$ to obtain a sequence back in $\ell^{2}.$

For $\phi, \psi \in L^{\infty}(\mathbb{T})$ the well-known identities
\begin{eqnarray}\label{talpro}
T(\phi \psi) & =& T(\phi)T(\psi) +H(\phi)H({\tilde \psi})\\
H(\phi \psi) & =& T(\phi)H(\psi) +H(\phi)T({\tilde \psi}),
 \end{eqnarray}
 where $\tilde{\phi}( e^{i\,\theta}) = \phi (e^{-i\,\theta}),$
are the basis of almost everything that follows.  
They can be proved by simply analyzing the matrix entries of the operators and using convolution to find the Fourier coefficients of products. They are also found in \cite{BS06}, Proposition 2.14.

It follows from these identities that if  $\psi_{-}$ and $\psi_{+} $ have the property that all their Fourier coefficients vanish for $k > 0$ and $k < 0$, respectively, then 
\begin{eqnarray*}
T(\psi_{-} \phi \psi_{+} )& =& T(\psi_{-} )T(\phi) T( \psi_{+} ) ,\\
& &\\
H(\psi_{-} \phi \tilde{\psi}_{+} )& =& T(\psi_{-} )H(\phi) T( \psi_{+} ).
\end{eqnarray*} 

Now suppose that the conditions of the Szeg\H{o}-Widom theorem hold. Then there exists a well-defined logarithm of $\phi$ that is in $B$. We can then  write $\log \phi$ as a sum of two functions, one with Fourier coefficients that vanish for positive indices, and one with Fourier coefficients that vanish for negative indices. If we then exponentiate these two functions, the factors (and their inverses) have the same Fourier properties and are in $B$. Thus we see that
\[ \phi = \phi_{-}\phi_{+}, \,\,\, \phi^{-1} = \phi_{-}^{-1} \phi_{+}^{-1} .\]  This factorization is known as a Wiener-Hopf factorization. 

We need one more ingredient before stating the BOCG identity. This is the notion of a Fredholm determinant, which is defined for any operator of the form $I + K$ where $K$ is a trace class operator. A compact operator $K$ is trace class if and only if
the eigenvalues $s_{i}$ of $(K^{*}K)^{1/2}$ satisfy $\sum_{i=0}^{\infty} s_{i} < \infty.$  We define the trace norm of $K$ by
\[ ||K||_{1} = \sum_{i=0}^{\infty} s_{i}.\] 
This set of operators is a closed ideal in the algebra of all bounded operators. 

For a trace class operator with discrete eigenvalues $\lambda_{i}$ one can show that  
\[ \sum_{i=0}^{\infty}|\lambda_{i}| < \infty \] and thus
\[ \det (I + K) = \prod_{i=0}^{\infty} (1+ \lambda_{i}) \]
is well defined and the product is called the Fredholm determinant.
Proofs of the above statements can be found in \cite{GK}.

The BOCG identity can now be stated as follows.
\begin{theorem}\label{BOCG} 
Suppose $\phi$ is in $B,$ does not vanish, and has winding number zero. Then
\begin{eqnarray*}
D_{n}(\phi) &=& G(\phi)^n \,\,E(\phi) \cdot 
\det\left(I-H(z^{-n}\phi_{-}\phi_{+}^{-1})H(\tilde{\phi}_{-}^{-1} \tilde{\phi}_{+}z^{-n})\right)
\end{eqnarray*}
In the above $z = e^{i\theta}.$
\end{theorem}

We will not prove this here, but refer to past proofs \cite{BW, Bot, GC, BO}. We note that the proofs in \cite{BW, Bot} show that the operator $$H(z^{-n}\phi_{-}\phi_{+}^{-1})H(\tilde{\phi}_{-}^{-1} \tilde{\phi}_{+}z^{-n})$$ is trace class, and thus the determinant on the right hand side of the identity is defined. 
Those proofs also show that the above product of Hankel operators tends to zero in the trace norm which then implies Szeg\H{o}'s asymptotic result. 

\section{Day's result}

We assume that our rational function $\phi$ has no multiple zeros and can be written in the form

$$\phi(e^{i\theta} )= c_{0} \prod_{j = 1}^{p}(e^{i\theta} - r_{j})\prod_{j=1}^{h}( 1 - \frac{e^{i\theta}}{\rho_{j}})^{-1}\prod_{j = 1}^{k}(e^{i\theta}-\delta_{j})^{-1},$$
 where $|\rho_{j}| > 1$, $|\delta_{j}| < 1$, $c_{0} \neq 0$ is a constant, and the $r_{j}$ are the distinct zeros of $\phi.$ Note that if we replace $e^{i\theta}$ by $z$ in the above expression, then $\phi(z)$ is a quotient of polynomials in $z.$ Day's result says the following. 
 
 \begin{theorem}
 Let $\phi$ be given as above. Then
 \begin{equation}
 D_{n}(\phi) = 0 \,\,\mbox{if}\,\, \,p < k, 
 \end{equation}
\begin{equation} \label{Day}
D_{n}(\phi) = (-1)^{(p-k)(n+1)}\sum_{M} A_{M}r^{n}_{M} \,\,\,\mbox{if} \,\,\,p\geq k,
\end{equation}
where the sum in (\ref{Day}) is taken over all subsets $M \subset \{1, \dots, ,p\}$ of cardinality $k$ and 
 $$r_{M} = c_{o} \prod_{j \in M^{c}} r_{j}, \,\,\,A_{M} = \prod_{\substack {j \in M^{c}, \,\alpha \in K \\ \beta \in H, \,i \in M}} \frac{(r_{j} - \delta_{\alpha})(\rho_{\beta} - r_{i})}{(\rho_{\beta} - \delta_{\alpha})(r_{j} - r_{i})},$$
with $M^{c}$ the complement of $M$, $K = \{1, \dots, ,k\}$, $H = \{1, \dots, ,h\}$.    \end{theorem}

The proof of this using the BOCG identity was done in \cite{BC}. However, the notation was very different from the one that Day used, but it is the notation commonly used in number theory applications for computing averages over the classical groups.  To describe the result, we need a few definitions.

 For sets of complex numbers $A$ and $B$ define
\begin{equation} Z(A,B)=\prod_{a\in A\atop b\in B}(1-ab)^{-1}. \label{Z-def}
\end{equation}
Furthermore,  let
$$Z(A,B;C,D)=\frac{Z(A,B)Z(C,D)}{Z(A,D)Z(B,C)}.$$
 
For sets $A$ and $T$ and a subset $U \subset A$,  we let $A-U+T^{-1}$ be the union of $A - U$ and the set of inverses of the elements of $T$. Finally, for a positive integer $n$, define $U^{n} := \prod_{u\in U}u^{n}.$

\begin{theorem}\label{bc}
Suppose that $p \geq k$ and let 
\[\phi(e^{i \theta}) = \prod_{i=1}^{k}\frac{(1 - a_{i}e^{-i \theta})}{(1 - c_{i}e^{-i \theta})}\prod_{i = 1}^{p - k}(1 - b_{i}e^{i \theta}))\prod_{i = 1}^{h}(1 - d_{i}e^{i \theta})^{-1},
\] with $|c_{i}|, |d_{i}| < 1$ and let 
$A = \{a_{1}, a_{2}, \dots, a_{k}\}, B =  \{b_{1}, b_{2}, \dots, b_{p - k}\} , \\ C = \{c_{1}, c_{2}, \dots, c_{k}\}, D =  \{d_{1}, d_{2}, \dots, d_{h}\}$. Then 
$$ D_{n}(\phi )=  \sum_{S\subset A ,T\subset B\atop |S|=|T|}S^n T^n   Z(A-S+T^{-1},B-T+S^{-1};C,D).
$$
\end{theorem}
These formulas look different, especially since the powers seem to be different. But they are really equivalent. 

Start with
$$\phi(e^{i\theta} )= c_{0} \prod_{j = 1}^{p}(e^{i\theta} - r_{j})\prod_{j=1}^{h}( 1 - \frac{e^{i\theta}}{\rho_{j}})^{-1}\prod_{j = 1}^{k}(e^{i\theta}-\delta_{j})^{-1},$$
and factor $e^{i\theta}$ out $k$ times from the numerator and denominator to get
$$\phi(e^{i\theta}) = c_{0} \prod_{j = 1}^{k}(1 - r_{j}e^{-i\theta})\prod_{j = k+1}^{p}(e^{i\theta} - r_{j})\prod_{j=1}^{h}( 1 - \frac{e^{i\theta}}{\rho_{j}})^{-1}\prod_{j = 1}^{k}(1-\delta_{j}e^{-i\theta})^{-1}.$$
Next factor out $\prod_{j = k+1}^{p}(- r_{j}).$ What remains is
$$\phi(e^{i\theta}) = c_{0} \prod_{j = k+1}^{p}(- r_{j})\prod_{j=1}^{k}\frac{(1 - r_{j}e^{-i\theta})}{(1-\delta_{j}e^{-i\theta})}\prod_{j = k+1}^{p}( 1 - e^{i \theta}/r_{j})\prod_{j=1}^{h}( 1 - \frac{e^{i\theta}}{\rho_{j}})^{-1}.$$ 
To see how the powers in Day's expansion agree, notice that if we use Theorem \ref{bc} to compute the determinant, we have a factor of 
$ (\prod_{j = k+1}^{p}(- r_{j}))^{n}$ times a product of $(r_{j})^{n}$ taken from a set $S$ and an equal number of $(1/r_{j})^{n}$ taken from a set $T.$ All possible products from a subset defined by a set $M$ can be obtained this way. We leave it to the reader to check that the coefficients in the two formulas agree.

While the proof in \cite{BC} for Theorem \ref{bc} uses the Fredholm expansion, there are four other known proofs. The first [CFZ] was motivated by physics and uses supersymmetry; the second [BG] uses symmetric function theory and representation theory. The third [CFS] uses  orthogonal polynomial methods from random matrix theory. We also point out that the proof in \cite{BC} required $p = 2k$ and $h = k.$ The proof easily holds without these restrictions or by allowing some of the parameters to be zero. 

The proof of Day's formula using the Fredholm expansion is straight-forward from the BOCG identity.  We will not reproduce all of the proof here, but give the general idea. If we let
\[\phi(e^{i \theta}) = \prod_{i=1}^{k}\frac{(1 - a_{i}e^{-i \theta})}{(1 - c_{i}e^{-i \theta})}\prod_{i = 1}^{p - k}(1 - b_{i}e^{i \theta}))\prod_{i = 1}^{h}(1 - d_{i}e^{i \theta})^{-1},
\] then since the determinants are polynomials in $a_{i}$ and $b_{i}$ we may assume that $|a_{i}|, |b_{i}| < 1.$  Thus we may apply the identity. We find that $G(\phi) = 1$ and by direct computation that $E(\phi) = Z(A, B; C, D).$

So what is left to do is to compute
\[\det\left(I-H(z^{-n}\phi_{-}\phi_{+}^{-1})H(\tilde{\phi}_{-}^{-1} \tilde{\phi}_{+}z^{-n})\right).\]
Recall that if $K$ has a discrete kernel $K(i,j)$ then the determinant of $I - K$ is given by the sum
\[1 +\sum_{l = 1 }^{\infty} \frac{(-1)^{l}}{l!} \left(\sum_{i_{1}, i_{2}, \cdots, i_{l}} \det K(i_{j}, i_{k})_{1 \leq j,k \leq l}\right). \]
The first term here is simply the trace of the operator $K$. The second term is a two-by-two determinant, and so on. 

Our next step is to compute this expansion and thus we need to find the matrix entries of $$H(z^{-n}\phi_{-}\phi_{+}^{-1})H(\tilde{\phi}_{-}^{-1} \tilde{\phi}_{+}z^{-n}).$$

Using complex analysis it is not hard to show that the $l$th Fourier coefficient of $\tilde{\phi}_{-}^{-1} \tilde{\phi}_{+}$ is given by
$ \sum_{j = 1}^{k}\al_{j}a_{j}^{l-1}$ and the $l$th  Fourier coefficient of $\phi_{-}\phi_{+}^{-1}$ is given by
$\sum_{j = 1}^{p-k}\beta_{j}b_{j}^{l-1},$ where $\al_{j}$ and $\beta_{j}$ are constants not depending on $l.$ This shows that the $i,j $ entry of the product of the above Hankel operators is a sum of terms of the form 
\[ \frac{\beta_{g}\al_{h}b_{g}^{n+i}a_{h}^{j+n}}{(1-b_{g}a_{h})}.\] When we substitute this into the formula for the Fredholm determinant we can expand the determinant by the columns. It turns out that only unique sets of $a_{j}$'s and $b_{h}$'s contribute, and these determinants are elementary to compute. For each $l \times l$ determinant we only have non-zero terms if $l \leq \min\{ k, p-k\},$ and if this is the case, then these yield the product factors in Theorem \ref{bc}. For example, a set $S$ with cardinality one will come from the first term, which is just a trace. A set of cardinality two will come from the two-by-two determinant and so on. The details are in \cite{BC}. 

\section{analogue of Day's formula for T + H}

Our goal is to compute the analogue of Day's formula for the determinant of the sum of finite Toeplitz plus Hankel matrices. These are finite matrices denoted by $T_{n}(\phi) + H_{n}(\phi),$ and with $i,j$ entry given by $$\phi_{i-j} + \phi_{i+j+1}.$$ In certain special cases this 
was also done in \cite{BC}. We begin by stating that result.

Let $A=\{a_1,\dots,a_k\}$ and define
\begin{equation} \label{ZOS-def}
   \begin{split} 
       Z_S(A) &=\prod_{1\le i\le j\le k}(1-a_i a_j)^{-1}, \\
       Z_O(A) &=\prod_{1\le i< j\le k}(1-a_i a_j)^{-1}, \\
       Z_O(A;C) &=\frac{Z_O(A)Z_S(C)}{Z(A,C)}.
    \end{split}
\end{equation}
\begin{theorem}\label{t3} Let $A = \{a_{1}, a_{2}, \dots, a_{k}\}$ and let $C = \{c_{1}, a_{2}, \dots, c_{k}\}$ where $|c_i|  <1$. Let
 \[\phi(\theta) =  \prod_{i=1}^{k}\frac{(1 - a_{i}e^{i \theta})(1 - a_{i}e^{-i \theta})}{(1 - c_{i}e^{i \theta})(1 - c_{i}e^{-i \theta})}.\]  
 Then 
 \[ \det (T_{n}(\phi) + H_{n}(\phi))
 = \prod_{i=1}^{k}\frac{(1+c_{i})}{(1+a_{i})}\sum_{U\subset A} U^nZ_O(A-U+U^{-1};C).\]
\end{theorem}

The function $\phi$ in the above theorem is ``even,'' that is, $\phi(e^{i\theta}) = \phi(e^{-i\theta}).$ When computing certain averages over cosets of classical groups, one encounters exactly the situation above with even symbols. 

Our goal is to extend this to the non-even setting and fortunately, the analogue of the BOGC identity is known and was proved in \cite{BE4}.

Let 
\[ P_{n} : \sum_{k = 0}^{\infty}\phi_{k}e^{ik\theta} \mapsto \sum_{k = 0}^{n-1}\phi_{k}e^{ik\theta} \] and let $Q_{n} = I - P_{n}.$
\begin{theorem}\label{BE} Suppose $\phi$ satisfies the conditions of the Szeg\H{o}-Widom theorem and in addition $H(\phi)$ and $H(\tilde{\phi})$ are trace class. Let $b = \log \phi$ (any choice of continuous logarithm will do). Then
\begin{eqnarray*}
\det(T_{n}(\phi) + H_{n}(\phi)) &=& G(\phi)^n \,\,E(\phi) \cdot 
\det\left(I +Q_{n}KQ_{n}\right)
\end{eqnarray*}
where $G(\phi)$ is as before, and
\[ E(\phi) = \exp  \tr \left( H(b) - \frac{1}{2}H(b)^{2} + H(b)H(\tilde{b})\right),\] 
\[ K = H(\psi)(T(\psi^{-1}) - H(\tilde{\psi}^{-1})),\] 
\[ \psi = \phi_{-}\phi_{+}^{-1} \tilde{\phi}_{+}^{-1}.\]
\end{theorem}

We will refer to the above identity as the BE identity, as was done in \cite{BC}.
\subsection{The simplest example}

Although the application of the above theorem to finding an exact formula for the determinants is straightforward, it is a bit messy. Thus we
begin with an example. We let 
\begin{equation}   \label{phi_k=1}
\phi(e^{i\theta}) = \frac{(1 - ae^{-i\theta})(1 - b e^{i\theta})}{(1 - c e^{-i\theta})(1 - de^{i\theta})},
\end{equation} 
and we suppose that $a, b, c$ and $d$ all have absolute value less than one. It is then known and not hard to check that the symbol $\phi$ satisfies the conditions of the above theorem. We find
 $G(\phi) = 1$ and $$E(\phi) = \frac{(1-b)(1+d)(1-cb)(1-ad)}{(1-bd)(1-ab)(1-cd)}.$$  
 
 \subsubsection{Computing the entries of $K$ for the example}
 
 The interesting computation is finding the entries of $K.$
 We have
 \[ \psi(e^{i\theta}) =  \frac{(1 - ae^{-i\theta})(1 - d e^{i\theta})(1 - d e^{-i\theta})}{(1 - c e^{-i\theta})(1 - be^{i\theta})(1 - be^{-i\theta})}.\]
Using simple contour integration we find that the following formulas hold and serve to define $\al_{b}, \al_{d1}, \al_{d2}, $ and $\al_{a}$:
\begin{enumerate}
\item $\psi_{k} = \frac{(1-ab)(1-db)(b-d)}{(1-cb)(1-b^{2})}b^{k-1} = \al_{b}b^{k-1},\,\,\,k>0$
\vspace*{.1in}
\item $\psi^{-1}_{k} = \frac{(1-cd)(1-bd)(d-b)}{(1- ad)(1-d^{2})}d^{k-1} =\al_{d1}d^{k-1}, \,\,\,k > 0$
\vspace*{.1in}
\item $\psi^{-1}_{-k} = \frac{(1-bd)(d-c)(d-b)}{(d-a)(1-d^{2})}d^{k-1} +\frac{(1-ab)(a-c)(a-b)}{(1-da)(a-d)}a^{k-1}=\al_{d2}d^{k-1} + \al_{a}a^{k-1}, \,\,\,k > 0$
\vspace*{.1in}
\item$\psi_{0}^{-1} = (\al_{d1} + b)/d$
\vspace*{.1in}
\item $\tilde{\psi}_{k}^{-1} = \psi_{-k}^{-1}.$
\end{enumerate}
\subsubsection{Computing the Fredholm determinant for the simple example}

Our goal now is to compute the Fredholm determinant 
\[\det\left(I+ Q_{n}KQ_{n}\right).
\]
The $j,l$ entry of $H(\psi)$ is $\al_{b}b^{j_+l}.$
Note that $H(\psi)$ is a rank one operator and thus only one term in the Fredholm expansion needs to be evaluated. 
The $l,k$ entry of $T(\psi^{-1}) - H(\tilde{\psi}^{-1})$ is $\psi^{-1}_{l-k} - \tilde{\psi}^{-1}_{l+k+1} .$
From this it follows that the $j,k$ entry of $K$ is 
\[ \al_{b}b^{j}\sum_{l=0}^{k-1}\left(\al_{d2}b^{l}d^{k-l-1} + \al_{a}b^{l}a^{k-l-1}\right) \,+ \,\al_{b}b^{j+k}(\al_{d1} + b)/d \]
\[
+ \al_{b}b^{j}\sum_{l=k+1}^{\infty}b^{l}\al_{d1}d^{l-k-1}\,-\,\al_{b}b^{j}\sum_{l=0}^{\infty}(\al_{d2}b^{l}d^{l+k} + \al_{a}b^{l}a^{l+k}),\]
or
\[\al_{b}\al_{d2}b^{j}d^{k-1}\frac{1 - (b/d)^{k}}{1- b/d} + \al_{b}\al_{a}b^{j}a^{k-1}\frac{1-(b/a)^{k}}{1- b/a}+\al_{b}\al_{d1}b^{j+k}\frac{1}{d(1-bd)} \]\[+ \al_{b}b^{j+k}b/d - \al_{b}\al_{d2}b^{j}d^{k}\frac{1}{1- bd} -\al_{b}\al_{a}b^{j}a^{k}\frac{1}{1- ab}\]
which is the same as
\[\al_{b}\al_{d2}b^{j}d^{k}\frac{(1+b)(1-d)}{(d-b)(1-bd)} + \al_{b}\al_{a}b^{j}a^{k}\frac{(1+b)(1-a)}{(a-b)(1-ab)}\]
\[+ \al_{b}b^{j+k}\left(-\al_{d2}\frac{1}{d-b} - \al_{a}\frac{1}{a-b} +  \al_{d1}\frac{1}{d(1-bd)} + \frac{b}{d}\right)\]
or
\[\al_{b}\al_{d2}b^{j}d^{k}\frac{(1+b)(1-d)}{(d-b)(1-bd)} + \al_{b}\al_{a}b^{j}a^{k}\frac{(1+b)(1-a)}{(a-b)(1-ab)},\] since the last term (coefficient of $b^{j+k}$) adds up to zero.

Because $H(\psi)$ is rank one, we compute the trace of $Q_{n}KQ_{n}$ and find that it is
\[  b^{n}d^{n}\frac{(1-ab)(b-d)(d-c)}{(d-a)(1+d)(1-b)(1-cb)} + b^{n}a^{n}\frac{(1-a)(a-c)(1-bd)(b-d)}{(a-d)(1-ad)(1-cb)(1-b)}.\]

This gives a final answer for the determinant:
\begin{equation} \label{det_formula_k=1}
\begin{split}
\det(T_{n}(\phi) + H_{n}(\phi))  &=\frac{(1-b)(1+d)(1-cb)(1-ad)}{(1-bd)(1-ab)(1-cd)} \\
&+ b^{n}d^{n}\frac{(1-ad)(b-d)(d-c)}{(d-a)(1-bd)(1-cd)} \\
&+ b^{n}a^{n}\frac{(1-a)(a-c)(1+d)(b-d)}{(a-d)(1-ab)(1-cd)}.
\end{split}
\end{equation}

\subsection{Computing the entries of $K$ for the general case}

Now we do the general case.

 From now on we assume that the parameters $a_{1}, \dots, a_{k},$ $b_{1}, \dots, b_{k},$ $d_{1}, \dots, d_{k},$ 
are all distinct. Later we will relax this restriction.

Let 
\[\phi(e^{i\theta}) = \prod_{i=1}^{k}\frac{(1 - a_{i}e^{-i\theta})(1 - b_{i} e^{i\theta})}{(1 - c_{i} e^{-i\theta})(1 - d_{i}e^{i\theta})},\] and assume that all parameters have absolute value less than one so that the BE identity holds.
First,
\[ E(\phi) = 
\frac{\prod_i( 1 -b_{i})(1 +d_{i})\prod_{ i <j }(1- b_{i}b_{j})(1- d_{i}d_{j})\prod_{i,j}(1- a_{i}d_{j})(1- b_{i}c_{j})}{\prod_{i,j}(1- b_{i}d_{j})(1- a_{i}b_{j})(1- c_{i}d_{j})} .
\]

Then we have
 \[ \psi(e^{i\theta}) =  \prod_{i=1}^{k}\frac{(1 - a_{i}e^{-i\theta})(1 - d_{i} e^{i\theta})(1 - d_{i} e^{-i\theta})}{(1 - c_{i}e^{-i\theta})(1 - b_{i}e^{i\theta})(1 - b_{i}e^{-i\theta})}.\] Our next goal is compute the appropriate Fourier coefficients. 
 
 The following definitions help to express the Fourier coefficients concisely.

 \begin{enumerate}
\item $ \al_{b_{j}}\,=\,\prod_{i=1}^{k}\frac{(1-a_{i}b_{j})(1-d_{i}b_{j})(b_{j}-d_{i})}{(1-c_{i}b_{j})(1-b_{i}b_{j})}\prod_{i\neq j}\frac{1}{b_{j}-b_{i}}$
 \vspace*{.1in}
 \item $\al_{d^{+}_{j}}\, =\,\prod_{i=1}^{k}\frac{(1-c_{i}d_{j})(1-b_{i}d_{j})(d_{j}-b_{i})}{(1- a_{i}d_{j})(1-d_{i}d_{j})}\prod_{i\neq j}\frac{1}{d_{j}-d_{i}}$
  \vspace*{.1in}
\item $\al_{d^{-}_{j}}\, =\,\prod_{i=1}^{k}\frac{(1-b_{i}d_{j})(d_{j}-c_{i})(d_{j}-b_{i})}{(d_{j}-a_{i})(1-d_{j}d_{i})}\prod_{i\neq j}\frac{1}{d_{j}-d_{i}}$
  \vspace*{.1in}
  \item $\al_{a_{j}} \, = \,\prod_{i=1}^{k}\frac{(1-a_{j}b_{i})(a_{j}-c_{i})(a_{j}-b_{i})}{(1-d_{i}a_{j})(a_{j}-d_{i})}\prod_{i\neq j}\frac{1}{a_{j}-a_{i}}$
 \end{enumerate}  

The Fourier coefficients are given by:

 \begin{enumerate}
\item $\psi_{n} = \sum_{j=1}^{k} \al_{b_{j}}b_{j}^{n-1},\,\,\,n > 0$
\vspace*{.1in}
\item $\psi^{-1}_{n} = \sum_{j=1}^{k}\al_{d^{+}_{j}}d_{j}^{n-1}, \,\,\,n > 0$
\vspace*{.1in}
\item $\psi^{-1}_{-n} = \sum_{j=1}^{k}(\al_{d^{-}_{j}}d_{j}^{n-1} + \al_{a_{j}}a_{j}^{n-1}), \,\,\,n > 0$
\vspace*{.1in}
\item$\psi_{0}^{-1} =  \sum_{j=1}^{k}(\al_{d^{-}_{j}}d_{j}^{-1} + \al_{a_{j}}a_{j}^{-1}) + \prod_{i=1}^{k}\frac{c_{i}b_{i}}{a_{i}d_{i}}$
\vspace*{.1in}
\item $\tilde{\psi}_{n}^{-1} = \psi_{-n}^{-1}.$
\end{enumerate}

For the matrix entries of $K$ we have
\[ K(g,h) = \sum_{l=0}^{\infty}\psi_{g+l+1}( \psi^{-1}_{l-h} - \psi^{-1}_{-l-h-1}) .\] We consider this in pieces. 
First we consider $$\sum_{l=0}^{\infty}\psi_{g+l+1}\, \psi^{-1}_{l-h}$$ which we break into three pieces depending on whether $l < h,\,\, l = h,$ or $l >h.$
\[ \sum_{l=0}^{h-1}\psi_{g+l+1} \psi^{-1}_{l-h} = \sum_{l=0}^{h-1}\left( \sum_{j=1}^{k} \al_{b_{j}}b_{j}^{g+l}\sum_{j=1}^{k}
(\al_{d^{-}_{j}}d_{j}^{h-l-1} + \al_{a_{j}}a_{j}^{h-l-1}) \right)\]
 which becomes
 \[\sum_{i,j = 1}^{k}\left( \frac{ \al_{b_{i}}\al_{d^{-}_{j}} {b_{i}^{g}}\,d_{j}^{h}}{d_{j}-b_{i}}  - \frac{ \al_{b_{i}}\al_{d^{-}_{j}} {b_{i}^{g+h}}}{d_{j}-b_{i}}\right) + \sum_{i,j=1}^{k}\left( \frac{ \al_{b_{i}}\al_{a_{j}} {b_{i}^{g}}\,a_{j}^{h}}{a_{j}-b_{i}}  - \frac{ \al_{b_{i}}\al_{a_{j}} {b_{i}^{g+h}}}{a_{j}-b_{i}}\right)\]
 The second is the $l=h$ term which is 
 \[\sum_{j=1}^{k} \al_{b_{j}}b_{j}^{g+h} \left(\sum_{j=1}^{k}(\al_{d^{-}_{j}}d_{j}^{-1} + \al_{a_{j}}a_{j}^{-1}) + \prod_{i=1}^{k}\frac{c_{i}b_{i}}{a_{i}d_{i}}\right)\]
 or
 \[\sum_{i,j=1}^{k}b_{i}^{g+h} (\al_{b_{i}}\al_{d^{-}_{j}}d_{j}^{-1} + \al_{b_{i}}\al_{a_{j}}a_{j}^{-1}) +\sum_{i=1}^{k}\al_{b_{i}}b_{i}^{g+h}\prod_{i=1}^{k}\frac{c_{i}b_{i}}{a_{i}d_{i}}. \]
 The next term is $l >h$ and we have
 \[\sum_{i,j=1}^{k}\sum_{l = h+1}^{\infty} \al_{b_{i}}\al_{d^{+}_{j}} b_{i}^{g+l} d_{j}^{l-h-1} = \sum_{i,j=1}^{k}\al_{b_{i}}\al_{d^{+}_{j}}b_{i}^{g+h}\frac{b_{i}}{1- b_{i}d_{j}}.\]

 The final term is 
 \[-\sum_{i,j=1}^{k}\left( \sum_{l=0}^{\infty} \al_{b_{j}}b_{j}^{g+l}(\al_{d^{-}_{j}}d_{j}^{l+h} + \al_{a_{j}}a_{j}^{l+h})\right)\]
 or
 \[-\sum_{i,j=1}^{k}\left( \al_{b_{i}}b_{i}^{g}\left(\al_{d^{-}_{j}}d_{j}^{h}\frac{1}{1-b_{i}d^{-}_{j}} + \al_{a_{j}}a_{j}^{h}\frac{1}{1-b_{i}a_{j}} \right)\right).\]
 Now we combine terms the following results. 
 The terms with $b_{i}^{g}d_{j}^{h}$ reduce to
 \[\sum_{i,j=1}^{k} \al_{b_{i}}\al_{d^{-}_{j}}b_{i}^{g}d_{j}^{h}\frac{(1+b_{i})(1-d_{j})}{(d_{j}-b_{i})(1-b_{i}d_{j})} . \]
 The terms involving $b_{i}^{g}a_{j}^{h}$ reduce to
 \[\sum_{i,j=1}^{k} \al_{b_{i}}\al_{a_{j}}b_{i}^{g}a_{j}^{h}\frac{(1+b_{i})(1-a_{j})}{(a_{j}-b_{i})(1-b_{i}a_{j})} . \]
 The term involving $b_{i}^{g+h}$ is
 \begin{equation} \label{vanishing-term}
 \al_{b_{i}}b_{i}\sum_{j=1}^{k}\left(- \al_{d^{-}_{j}}\frac{1}{d_{j}(d_{j}-b_{i})} - \al_{a_{j}}\frac{1}{a_{j}(a_{j}-b_{i})}   + \al_{d^{+}_{j}}\frac{1}{1- b_{i}d_{j}}\right) +   \al_{b_{i}}\prod_{j=1}^{k}\frac{b_{j}c_{j}}{a_{j}d_{j}}.
 \end{equation}
 
In Appendix B we show that this term vanishes. Hence, what remains is that the $g,h$ entry of our operator $K$ is  \[\sum_{i,j=1}^{k} \al_{b_{i}}\al_{d^{-}_{j}}b_{i}^{g}d_{j}^{h}\frac{(1+b_{i})(1-d_{j})}{(d_{j}-b_{i})(1-b_{i}d_{j})} +
 \sum_{i,j=1}^{k} \al_{b_{i}}\al_{a_{j}}b_{i}^{g}a_{j}^{h}\frac{(1+b_{i})(1-a_{j})}{(a_{j}-b_{i})(1-b_{i}a_{j})} . \]

\subsection{Computing the terms in the Fredholm expansion in the general case}
Now that we know the entries of $K$ our next step is to compute the terms in the Fredholm expansion. We need some preliminaries first.

Let $S$ and $T$ be ordered sets of complex numbers and let $I = \{i_{1}, \dots, i_{l}\}.$ Define 
the $l \times l$ determinant by 
\[ D_{I}(\{t_{i}\}, \{s_{j}\}) = \det (t_{h}^{i_{g}} \,s_{h}^{i_{h}})_{ 1 \leq g, h \leq l} \] 
with $t_{i} \in T$ and $s_{i} \in S.$

Here is the $3 \times 3$ example for sets $\{t_{1}, t_{2}, t_{3}\}$ and $\{s_{1}, s_{2}, s_{3}\}$:

\[ D_{I}(\{t_{i}\}, \{s_{j}\}) = \left|\begin{array}{ccc} t_{1}^{i_{1}}s_{1}^{i_{1}} & t_{2}^{i_{1}}s_{2}^{i_{2}}& t_{3}^{i_{1}}s_{3}^{i_{3}}
 \\ t_{1}^{i_{2}}s_{1}^{i_{1}} & t_{2}^{i_{2}}s_{2}^{i_{2}}& t_{3}^{i_{2}}s_{3}^{i_{3}} \\
 t_{1}^{i_{3}}s_{1}^{i_{1}} & t_{2}^{i_{3}}s_{2}^{i_{2}}& t_{3}^{i_{3}}s_{3}^{i_{3}}\end{array}\right| .\]

The following two lemmas were proved in \cite{BC} and will be used in our computations.

\begin{lemma}
If $t_{i} = t_{j}$ for $i \neq j,$ then $D_{I}(\{t_{i}\}, \{s_{j}\})  = 0.$
\end{lemma}
This is easy to see since we can factor out $s_{h}^{i_{h}}$ from any column and the resulting determinant has two equal columns. 

\begin{lemma} For a positive integer $n,$
if $s_{i} = s_{j}$, for $i \neq j,$ then $ \sum_{i_{i}, \dots ,i_{l} \geq n} D_{I}(\{t_{i}\}, \{s_{j}\})  = 0.$
\end{lemma}

Now we continue computing the terms in the Fredholm expansion in the general case
Let $A = \{a_{1}, a_{2}, \cdots, a_{k}\}, $
 $B = \{b_{1}, b_{2}, \cdots, b_{k}\}, C = \{c_{1}, c_{2}, \cdots, c_{k}\}$ and $D = \{d_{1}, d_{2}, \cdots, d_{k}\}.$  
 Let $S \subset  E = A + D$ and $ T \subset B.$

The columns of the matrix $K(i_{j},i_{k})$ are sums of $2k^{2}$ column vectors since each entry of $K(j,k)$ is a sum of $2k^{2}$ entries. Expanding the determinant of $K(i_{j},i_{k})$ by columns yields $(2k^{2})^{l}$ terms, but by our lemmas, many of these vanish. Note the column vectors are determined by a choice of a $b_{i}$ and a choice of either an $a_{i}$ or $d_{i}.$ We are only left to deal with disjoint subsets of elements of $s_{i}$ from $A + D$ and $t_{i}$ from $B$. We will call those subsets $S$ and $T.$

Thus the $l \times l$ determinant of  $K(i_{j}, i_{k})$ is a sum of determinants, each of the form 
 \[ \prod_{j = 1}^{l}\al_{t_{j}}\al_{s_{j}}\frac{(1+t_{j})(1-s_{j})}{(s_{j}-t_{j})(1-t_{j}s_{j})}D_{I}(\{t_{i}\}, \{s_{i}\})\] 
 where the sum is taken over all ordered sets $S$ and $T$ of size $l,$
and  $\al_{s_{j}}$ is either $\al_{a_{j}}$ or $\al_{d_{j}^{-}}.$

We begin with two fixed ordered sets $S$ and $T$ and then consider the pair $S$ and $T^{'}$ where $T^{'}$ is a permutation of $T.$
 Then the contribution from all possible permutations of the $t_{j}$'s is 
 \[ \prod_{j = 1}^{l}\al_{t_{j}}\al_{s_{j}} (1+t_{j})(1-s_{j})\,\,(D_{I}(\{t_{i}\}, \{s_{i}\})\left(\sum_{\sigma}\sgn(\sigma)\frac{1}{\prod_{j=1}^{l}(s_{j}-t_{\sigma(j)})(1-t_{\sigma(j)}s_{j})}\right).\]
 Here $D_{I}$ represents the determinant $D_{I}(\{t_{i}\} ,\{s_{i}\})$ for the sets $S$ and $T$. This expression on the right is a Cauchy type determinant. It is evaluated in Appendix A and the evaluation yields
  \[ \prod_{j = 1}^{l}\al_{t_{j}}\al_{s_{j}} (1+t_{j})(1-s_{j}) \,\,(D_{I}(\{t_{i}\}, \{s_{i}\})\frac{\prod_{i < j}(t_{i}-t_{j})(s_{j}-s_{i})(1- t_{i}t_{j})(1-s_{i}s_{j})}{\prod_{i,j}(s_{i}-t_{j})(1-t_{i}s_{j})},\] which must be summed over $ i_{i},\dots, i_{l} \geq n.$
 
  First expand $D_{I}(\{t_{i}\}, \{s_{i}\}$ using the permutation definition. Then sum $\sum_{i_{i},\dots,i_{l} \geq n}D_{I}(\{t_{i}\}, \{s_{i}\})$. The result is $\prod_{j=1}^{l}t_{j}^{n}s_{j}^{n}$ times a Cauchy determinant given by 
  \[ \frac{\prod_{i < j}(t_{j} - t_{i})(s_{j} - s_{i})}{\prod_{i,j}(1-t_{i}s_{j})}.\]
  Thus the contribution after summing in the $l \times l$ Fredholm determinant expansion for the ordered set $S$ and all permutations of $T$ is 
  \[ \prod_{j=1}^{l}t_{j}^{n}s_{j}^{n}\prod_{j = 1}^{l}\al_{t_{j}}\al_{s_{j}} (1+t_{j})(1-s_{j}))\]\[\times \frac{\prod_{i < j}(t_{i}-t_{j})(s_{j}-s_{i})(1- t_{i}t_{j})(1-s_{i}s_{j})}{\prod_{i,j}(s_{i}-t_{j})(1-t_{i}s_{j})}\frac{\prod_{i < j}(t_{j} - t_{i})(s_{j} - s_{i})}{\prod_{i,j}(1-t_{i}s_{j})}.\]
 
 Notice that this expression does not depend on the order $S.$ Hence we have $l!$ occurrences of these which cancel with $1/l!$ in the Fredholm expression.  
 
 Now we are in a position to write down an answer. 
 
Recall that $\al_{s_{j}}$ is either $\al_{a_{j}}$ or $\al_{d_{j}^{-}}.$
However it can be written concisely as  
 $$\al_{s_{j}} = \frac{1}{s_{j} - \bar{s}_{j}} \prod_{i=1}^{k}\frac{(1-b_{i}s_{j})(s_{j}-c_{i})(s_{j}-b_{i})}{(1-s_{j}d_{i})}\prod_{\substack{i: a_{i}\neq s_{j}\\d_{i}\neq s_{j}}}\frac{1}{(s_{j}-a_{i})(s_{j}-d_{i})}$$
where $ \bar{s}_{j}= a_{j}$ if $s_{j} = d_{j}$ and $\bar{s}_{j} = d_{j}$ if $s_{j} = a_{j}.$

We no longer need to consider the order of the sets $S$ and $T$, and so we treat them as ordinary subsets of $E$ and $B.$

For the sets $S$ and $T$ the Fredholm expansion yields a term
\begin{multline} \label{ST} E(\phi)\prod_{i=1}^{l}t_{i}^{n}s_{i}^{n}\prod_{i = 1}^{l}\al_{t_{i}}\al_{s_{i} }(1+t_{j})(1-s_{j}))\\\times \frac{\prod_{i < j}(t_{i}-t_{j})(t_{j} - t_{i})(s_{i}-s_{j})^{2}(1- t_{i}t_{j})(1-s_{i}s_{j})}{\prod_{i,j}(s_{i}-t_{j})(1-t_{i}s_{j})^{2}}.\end{multline}
Recall that 
\[ E(\phi) = \frac{\prod_{i=1}^{k}( 1 -b_{i})(1 +d_{i})\prod_{1\leq i <j \leq k}(1- b_{i}b_{j})(1- d_{i}d_{j})\prod_{i,j}(1- a_{i}d_{j})(1- b_{i}c_{j})}{\prod_{ i,j }(1- b_{i}d_{j})(1- a_{i}b_{j})(1- c_{i}d_{j})}. \]

Thus (\ref{ST}) is
\[ \prod_{i=1}^{l}t_{i}^{n}s_{i}^{n}\prod_{i=1}^{k}( 1 -b_{i})(1 +d_{i})\prod_{1\leq i <j \leq k}(1- b_{i}b_{j})(1- d_{i}d_{j})\prod_{i = 1}^{l}(1+t_{i})(1-s_{i})\al_{t_{i}}\al_{s_{i}} \]
\[\times \frac{\prod_{i,j}(1- a_{i}d_{j})(1- b_{i}c_{j})}{\prod_{i,j}(1- b_{i}d_{j})(1- a_{i}b_{j})(1- c_{i}d_{j})} 
 \frac{\prod_{i < j\leq l}(t_{i}-t_{j})(t_{j} - t_{i})(s_{i}-s_{j})^{2}(1- t_{i}t_{j})(1-s_{i}s_{j})}{\prod_{1\leq i,j \leq l}(s_{i}-t_{j})(1-t_{i}s_{j})^{2}}\] 
where
\[ \al_{t_{j}} = \prod_{p=1}^{k}\frac{(1-a_{p}t_{j})(1-d_{p}t_{j})(t_{j}-d_{p})}{(1-c_{p}t_{j})(1-b_{p}t_{j})}\prod_{p\neq j}\frac{1}{t_{j}-b_{p}}\]
and 
\[\al_{s_{j}} = \frac{1}{s_{j} - \bar{s}_{j}} \prod_{p=1}^{k}\frac{(1-b_{p}s_{j})(s_{j}-c_{p})(s_{j}-b_{p})}{(1-s_{j}d_{p})}\prod_{\substack{p: a_{p}\neq s_{j}\\d_{p}\neq s_{j}}}\frac{1}{(s_{j}-a_{p})(s_{j}-d_{p})}.\]

Changing the $t_{i}$ notation back to $b_{i}$ we can see that many of the terms can be combined to show that (\ref{ST}) is equal to
\begin{gather*}
\prod_{b_{p} \in T}b_{p}^{n}\prod_{s_{p}\in S}s_{p}^{n}\prod_{i=1}^{k}( 1 -b_{i})(1 +d_{i})\prod_{1\leq i <j \leq k}(1- d_{i}d_{j})\prod_{b_{p} \in T}(1+b_{p})\prod_{s_{p}\in S}\frac{(1-s_{p}) }{s_{p} - \bar{s}_{p}} \\
 \times \prod_{\substack{b_{q}\in B-T\\a_{p} \in A} }\frac{1}{(1 - a_{p}b_{q})} 
\prod_{\substack{b_{q}\in B-T\\d_{p} \in  D} }\frac{1}{(1 - d_{p}b_{q})} \prod_{\substack{c_{p} \in C \\ d_{q} \in D}}\frac{1}{(1 - c_{p}d_{q})}\prod_{\substack{s_{q}\in S\\d_{p}\in D}}\frac{1}{(1 - s_{q}d_{p})} \prod_{\substack{b_{p}\in T \\ s_{q}\in S}}\frac{1}{(1 - b_{p}s_{q})} \\
\times \prod_{\substack{s_{p} \neq a_{q}\\s_{p}\in S\\a_{q}\in A}}\frac{1}{s_{p} - a_{q}}\prod_{\substack{s_{p} \neq d_{q}\\s_{p} \in S\\d_{q}\in D}}\frac{1}{s_{p}- d_{q}}
  \prod_{\substack{b_{p}\in B-T\\c_{q}\in C}}(1-b_{p}c_{q}) \prod_{\substack{a_{p}\in A\\d_{q}\in D}}( 1- a_{p}d_{q})
\prod_{\substack{b_{q}\in T\\d_{p}\in D}}(b_{q}-d_{p}) \\
\times\prod_{\substack{s_{q}\in S\\c_{p}\in C}}(s_{q}-c_{p})\prod_{\substack{b_{q}\in T\\b_{p}\in B-T}}\frac{1}{(b_{q}-b_{p})}\prod_{\substack{s_{q}\in S\\b_{p}\in B-T}}(s_{q}-b_{p})\prod_{\substack{p < q \\b_{p},b_{q} \in B-T}}(1 - b_{q}b_{p}) \\
 \times \prod_{b_{q} \in T}\frac{1}{1-b_{q}^{2}} \prod_{\substack{p < q \\s_{p},s_{q} \in S}}(1 - s_{q}s_{p})(s_{q}-s_{p})^{2}\prod_{\substack{b_{p}\in B-T\\s_{q}\in S}}(1-s_{q}b_{p}).
\end{gather*}

 With additional simplification (\ref{ST}) becomes
 \[ (-1)^{l(l-1)/2}\prod_{b_{p} \in T}b_{p}^{n}\prod_{s_{p}\in S}s_{p}^{n}\prod_{i=1}^{k}(1 +d_{i})\prod_{1\leq i <j \leq k}(1- d_{i}d_{j})
 \prod_{s_{q}\in S}(1-s_{q}) \prod_{b_{q}\in B-T}( 1- b_{q}) \]
  \[\prod_{\substack{p < q \\b_{p},b_{q} \in B-T}}(1 - b_{q}b_{p})\prod_{\substack{b_{q}\in T\\b_{p}\in B-T}}\frac{1}{(b_{q}-b_{p})}\prod_{\substack{p < q \\s_{p},s_{q} \in S}}(1 - s_{q}s_{p})\prod_{\substack{s_{q} \in S\\s_{p}\in E-S}}\frac{1}{(s_{q} - s_{p})}\]
  \[ \times \frac
 { \prod_{\substack{b_{p}\in B-T\\c_{q}\in C}}(1-b_{p}c_{q}) \prod_{\substack{a_{p}\in A\\d_{q}\in D}}( 1- a_{p}d_{q})
\prod_{\substack{b_{q}\in T\\d_{p}\in D}}(b_{q}-d_{p})\prod_{\substack{s_{q}\in S\\c_{p}\in C}}(s_{q}-c_{p})
\prod_{\substack{s_{q}\in S\\b_{p}\in B-T}}(s_{q}-b_{p})}
{\prod_{\substack{b_{q}\in B-T\\a_{p} \in A + D - S} }(1 - a_{p}b_{q})\prod_{\substack{c_{p} \in C \\ d_{q} \in D}}(1 - c_{p}d_{q})\prod_{\substack{s_{q}\in S\\d_{p}\in D}}(1 - s_{q}d_{p}) \prod_{\substack{b_{p}\in T \\ s_{q}\in S}}(1 - b_{p}s_{q})}\]

\subsection{Final answer}

The next steps convert the above answer into something that is concise and an analogue of the number theory version of Day's result. To do this we use the following properties, each of which is easy to verify.

\begin{lemma} \label{Z-properties}
Let $Z(A,B), Z_{O}(A)$ be defined as before in (\ref{Z-def}) and (\ref{ZOS-def}). Then
\begin{align*} 
(1)&\; Z(A,B) = Z(B,A) \\
(2)&\; Z(A + B, C) = Z(A,C)Z(B,C) \\
(3)&\; Z_{O}(A + B) = Z_{O}(A)Z_{O}(B)Z(A,B) \\
(4)&\; \prod_{a \in A, b \in B} \frac{1}{a-b} = \prod_{a \in A} a^{-|B|} \prod_{a \in A, b \in B}\frac{1}{1-b/a} 
     = \prod_{a \in A} a^{-|B|} Z(A^{-1},B) \\
(5)&\; Z(A,B) = (-1)^{|A|\,|B|}\prod_{a \in A}a^{-|B|}\prod_{b \in B} b^{-|A|} Z(A^{-1},B^{-1}) \\
(6)&\; Z_O(A) =  (-1)^{|A|(|A|-1)/2}\prod_{a \in A} a^{-|A|+1}Z_{O}(A^{-1})
\end{align*}
\end{lemma}

Using properties (1)-(4) we can express (\ref{ST}) as
\[ (-1)^{l(l-1)/2}\prod_{b_{p} \in T}b_{p}^{n}\prod_{s_{p}\in S}s_{p}^{n}\prod_{i=1}^{k}(1 +d_{i})
 \prod_{s_{q}\in S}(1-s_{q}) \prod_{b_{q}\in B-T}( 1- b_{q}) \]
 \[\times \frac{\prod_{b_{q} \in T}b_{q}^{l-k}Z(T^{-1}, B-T)\prod_{s_{q} \in S}s_{q}^{l-2k}Z(S^{-1}, E-S)}
 {Z_{O}(D)Z_{O}(B-T)\,\,Z_{O}(S)}\]
 \[ \times \frac{Z(E - S,B-T)Z(C,D)Z(S,D)Z(T,S)}{Z(B-T, C)Z(A,D)\prod_{i}b_{i}^{-k}Z(T^{-1},D) \prod_{s_{q} \in S}s_{q}^{-k}Z(S^{-1}, C)\prod_{s_{q} \in S}s_{q}^{l-k}Z(S^{-1}, B-T)  }  \]

With properties (5) and (6) this can be rewritten as 
\[ (-1)^{l}\prod_{b_{p} \in T}b_{p}^{n}\prod_{s_{p}\in S}s_{p}^{n}\prod_{i=1}^{k}(1 +d_{i})
 \prod_{s_{q}\in S}(1-s_{q}) \prod_{b_{q}\in B-T}( 1- b_{q}) \]
 \[\times \frac{\prod_{b_{q} \in T}b_{q}^{l-k}Z(T^{-1}, B-T)\prod_{s_{q} \in S}s_{q}^{l-2k}Z(S^{-1}, E-S)}
 {Z_{O}(D)Z_{O}(B-T)\,\,\prod_{s_{q} \in S}s_{q}^{-l+1}Z_{O}(S^{-1})}\]
 \[ \times \frac{Z(E - S,B-T)Z(C,D)Z(D,D)Z(S,D)\prod_{b_{q} \in T}b_{q}^{-l}\prod_{s_{q} \in S}s_{q}^{-l}Z(T,^{-1}S^{-1})}{Z(B-T, C)Z(D,D)Z(A,D)\displaystyle\prod_{b_q \in T}b_{q}^{-k}Z(T^{-1},D) \displaystyle\prod_{s_{q} \in S}s_{q}^{-k}Z(S^{-1}, C)\displaystyle\prod_{s_{q} \in S}s_{q}^{l-k}Z(S^{-1}, B-T)  }.  \]

Then using properties (2) and (3) of Lemma \ref{Z-properties} and the fact that 
\[ \frac{Z(S,D)}{Z(A,D)Z(D,D)} = \frac{1}{Z(E - S, D)} \] the expression for (\ref{ST}) 
further simplifies to

\[(-1)^{l} \prod_{\substack{s_{i} \in S\\b_{i} \in T}} s_{i}^{n-1}b_{i}^{n} 
\prod_{d_{i} \in D}(1 + d_{i})\prod_{s_{i} \in S}(1 -s_{i})\prod_{b_{i} \in B-T}(1 -b_{i})\]
\[ \times \frac{Z_{S}(D)Z(C,D)Z(E - S +T^{-1}, B-T +S^{-1})}{Z(E - S +T^{-1},D)Z(B-T +S^{-1},C)Z_{O}(B-T +S^{-1})} .\]

This is the term in the Fredholm expansion corresponding to the sets $S$ and $T$. Summing over $S$ and $T$ gives the the full expansion for $\det(T_{n}(\phi) + H_{n}(\phi))$, which we summarize in the following theorem.
\begin{theorem}
Let $A = \{a_{1}, \dots, a_{k}\}$, $B =  \{b_{1}, \dots, b_{k}\}$, $C = \{c_{1}, \dots, c_{k}\}$, 
$D = \{d_{1}, \dots, d_{k}\}$, and $E = A + D$. Assume $|c_i|, |d_i| <1,$ and that $d_{1}, \dots, d_{k}$ are distinct. Let
\[\phi(e^{i \theta}) = \prod_{i=1}^{k}\frac{(1 - a_{i}e^{-i \theta})(1 - b_{i}e^{i \theta})}{(1 - c_{i}e^{-i \theta})(1 - d_{i}e^{i \theta})}. \]
Then 
\begin{multline*}
\det(T_{n}(\phi) + H_{n}(\phi))  \\
   =  \sum_{\substack{S \subset E\\ T \subset B\\ |S|  = |T|}}(-1)^{|S|} \prod_{\substack{s_{i} \in S\\b_{i} \in T}} s_{i}^{n-1}b_{i}^{n} 
\prod_{d_{i} \in D}(1 + d_{i})\prod_{s_{i} \in S}(1 -s_{i})\prod_{b_{i} \in B-T}(1 -b_{i}) \\
\times \frac{Z_{S}(D)Z(C,D)Z(E - S +T^{-1}, B-T +S^{-1})}{Z(E - S +T^{-1},D)Z(B-T +S^{-1},C)Z_{O}(B-T +S^{-1})}.
\end{multline*}
\end{theorem}

\begin{proof}
The only thing to mention is that the application of the identity in the above setting required that all of the parameters had absolute value less than one and that the $a_{i}, b_{i}, d_{i}$ parameters were distinct. The Fourier coefficients of $\phi$ are polynomials in $a_{i}$ and $b_{i}$ and therefore the determinant of $T_{n}(\phi) + H_{n}(\phi)$ is also. Thus, the above 
holds for all $a_{i}$ and $b_{i}.$
\end{proof}

\begin{corollary}
Suppose $A = B$ and $C = D.$ Then \[ \det (T_{n}(\phi) + H_{n}(\phi))
 = \prod_{i=1}^{k}\frac{(1+c_{i})}{(1+a_{i})}\sum_{S \subset A} \prod_{a_{i}\in S} a_{i}^{2n+1}Z_O(A-S+S^{-1};C).\]
\end{corollary}
\proof
First note that because of the factors $$\prod_{\substack{s_{q}\in S\\c_{p}\in C}}(s_{q}-c_{p})
\prod_{\substack{s_{q}\in S\\b_{p}\in B-T}}(s_{q}-b_{p})$$
we may assume that $S = T$ and that $S \subset A.$ Thus our sum is only over subsets of $A$. Letting $S = T$ we have 
\[  \sum_{S \subset A}(-1)^{|S|} \prod_{a_{i} \in S} a_{i}^{2n-1} 
\prod_{c_{i} \in C}(1 + c_{i})\prod_{a_{i} \in A}(1 -a_{i})\]
\[ \times \frac{Z_{S}(C)Z(C,C)Z(A + C - S +S^{-1}, A - S +S^{-1})}{Z(A + C - S +S^{-1},C)Z(A - S +S^{-1},C)Z_{O}(A - S +S^{-1})}.\]
\[= \sum_{S \subset A}(-1)^{|S|} \prod_{a_{i} \in S} a_{i}^{2n-1} 
\prod_{c_{i} \in C}(1 + c_{i})\prod_{a_{i} \in A}(1 -a_{i})\]
\[ \times \frac{Z_{S}(C)Z(C,C)Z(A - S +S^{-1}, A - S +S^{-1})Z(C, A - S +S^{-1})}{Z(A  - S +S^{-1},C)Z(C,C)Z(A - S +S^{-1},C)Z_{O}(A - S +S^{-1})}\]
\[ = \sum_{S \subset A}(-1)^{|S|} \prod_{a_{i} \in S} a_{i}^{2n-1} 
\prod_{c_{i} \in C}(1 + c_{i})\prod_{a_{i} \in A}(1 -a_{i})\]
\[ \times \frac{Z_{S}(C)Z(A - S +S^{-1}, A - S +S^{-1})}{Z(A  - S +S^{-1},C)Z_{O}(A - S +S^{-1})}\]
\[ = \sum_{S \subset A}(-1)^{|S|} \prod_{a_{i} \in S} a_{i}^{2n-1} 
\prod_{c_{i} \in C}(1 + c_{i})\prod_{a_{i} \in A}(1 -a_{i})\]
\[ \times \frac{Z_{S}(C)Z_{S}(A - S +S^{-1})}{Z(A  - S +S^{-1},C)}\]
\[ = \sum_{S \subset A}(-1)^{|S|} \prod_{a_{i} \in S} a_{i}^{2n-1} 
\prod_{c_{i} \in C}(1 + c_{i})\prod_{a_{i} \in A}(1 -a_{i})\]
\[ \times \frac{Z_{S}(C)Z_{O}(A - S +S^{-1})}{Z(A  - S +S^{-1},C)\prod_{a_{i}\in A - S}(1 -a_{i}^{2})\prod_{a_{i}\in  S}(1 -1/a_{i}^{2})}\]
\[= \prod_{i=1}^{k}\frac{(1+c_{i})}{(1+a_{i})}\sum_{S \subset A} \prod_{a_{i}\in S} a_{i}^{2n+1}Z_O(A-S+S^{-1};C).\]

\endproof

We end this section with some remarks about generalizing our main theorem. We required that the number of parameters $a_{i}, b_{i}, c_{i}$ and $d_{i}$ be the same to keep the proofs simpler. However the answer in the main theorem is independent of the number of parameters and is valid whenever it makes sense. 

There are also other classes of operators where one can apply the same techniques used in this paper, such as, for example, the difference of finite Toeplitz and Hankel matrices. The analogues of the BOCG identities can be found in \cite{BE4}, and then one would need to repeat the steps done in this paper.  

\section{Numerical examples, eigenvalues, and conjectures}

In order to check the derivation of the formula for the determinants of the finite Toeplitz plus Hankel matrices, we computed some small determinants for rational symbols exactly and then did a comparison with the formula of the previous section. The Fourier coefficients were exact rational numbers. Here are some examples.

\begin{example}
Let
\[ \phi(e^{i\theta}) = \frac{(1 - ae^{-i\theta})(1 - b e^{i\theta})}{(1 - c e^{-i\theta})(1 - de^{i\theta})}, \] 
where $a = 1/2, b=1/3, c=1/4, d=1/5$.  The determinant formula  (\ref{det_formula_k=1}) with $n=5$ gives
\[ \det(T_5(\phi) + H_5(\phi)) = \frac {51551341}{57712500} \approx 0.893244. \]
The Fourier coefficients of $\phi$, which can be computed with residues, are given by
\[\phi_k = \begin{dcases}   
   \frac{(1-a d) (d-b) d^{k-1}}{1-c d} & \text{if $k>0$}, \\
   \frac{(1-a d) (d-b)}{d (1-c d)}+\frac{b}{d} & \text{if $k=0$},\\
   \frac{(c-a) (1-b c) c^{| k| -1}}{1-c d} & \text{if $k < 0$}.
   \end{dcases}
 \]
 A computer calculation of $\det(T_5(\phi) + H_5(\phi))$, using exact arithmetic, gives the same rational number.
\end{example}

\begin{example}
Let
\[ \phi(e^{i\theta}) = \frac{(1 - ae^{-i\theta})(1 - b e^{i\theta})}{(1 - c e^{-i\theta})(1 - de^{i\theta})}, \] 
where $a = 2, b=1/3, c=1/4, d=1/5$.  The determinant formula  (\ref{det_formula_k=1}) with $n=5$ gives
\[ \det(T_5(\phi) + H_5(\phi)) = \frac {7571}{4617} \approx 01.63981. \]
 A computer calculuation of $\det(T_5(\phi) + H_5(\phi))$, using exact arithmetic, gives the same rational number.
\end{example}

One of Days' goals was to understand the eigenvalues of finite Toeplitz matrices with rational symbols as $n \rightarrow \infty$. He approached this by asking when the matrix
\[ T_{n}(\phi) - \lambda I = T_{n}( \phi - \lambda) \] 
could have determinant zero. Note that $\phi - \lambda$ is rational when $\phi$ is rational.  

Let $z = e^{i\theta}$ and write
\[ \phi(z) = \frac{f(z)}{g(z)} .\] Then
 \[ \phi(z)- \lambda  = \frac{f(z) - \lambda g(z)}{g(z)}. \] 

Suppose the degree of $f(z) - \lambda g(z)$ is $p$  and let the number of roots of $g(z)$ inside the unit circle be $k.$ 

We let $\,z_{i}(\lambda)\,\,i=1, \dots, p$ be the zeros of $\phi(z)- \lambda,$ ordered so that 
\[ |z_{1}(\lambda)| \leq |z_{2}(\lambda)| \leq \cdots \leq |z_{p}(\lambda)|.\]
Day argued that the determinant $D_{n}(\phi)$ could not be zero unless 
\[ |z_{k}(\lambda)|  =  |z_{k+1}(\lambda)|.\]
This defines a curve in the complex plane where the eigenvalues must accumulate.

To understand his reasoning recall
$$\phi(e^{i\theta} )= c_{0} \prod_{j = 1}^{p}(e^{i\theta} - r_{j})\prod_{j=1}^{h}( 1 - \frac{e^{i\theta}}{\rho_{j}})^{-1}\prod_{j = 1}^{k}(e^{i\theta}-\delta_{j})^{-1},$$
 where $|\rho_{j}| > 1,$ $|\delta_{j}| < 1$, and  $r_{j}$ are the zeros of $\phi.$ 
Then
$$D_{n}(\phi) = (-1)^{(p-k)(n+1)}\sum A_{M}r^{n}_{M}$$
where the sum is taken over all ${p \choose k}$ subsets  $M \subset \{1, \dots, ,p\}$ of cardinality $k$, and
$$r_{M} = c_{o} \prod_{j \in M^{c}} r_{j}, \,\,\,A_{M} = \prod_{\substack {j \in M^{c}, \,\alpha \in K \\ \beta \in H, \,i \in M}} \frac{(r_{j} - \delta_{\alpha})(\rho_{\beta} - r_{i})}{(\rho_{\beta} - \delta_{\alpha})(r_{j} - r_{i})},$$
with $M^{c}$ the complement of $M$, $K = \{1, \dots, ,k\}$, $H = \{1, \dots, ,h\}$.   

Thus, if 
\[ |z_{k+1}(\lambda)| > |z_{i}(\lambda)| \,\, i = 1, \dots, k\]
then the determinant of $T_n(\phi-\lambda)$ will not vanish for $n$ sufficiently large because \[ |z_{k+1}(\lambda) \,z_{k+2}(\lambda) \dots z_{p}(\lambda)|^{n} \] will be the dominating term, which corresponds to $M^c = \{k+1,\dots,p\}$.

For the other version of Day's formula, we start with 
\[\phi(e^{i \theta}) = \prod_{i=1}^{k}\frac{(1 - a_{i}e^{-i \theta})}{(1 - c_{i}e^{-i \theta})}\prod_{i = 1}^{p - k}(1 - b_{i}e^{i \theta}))\prod_{i = 1}^{h}(1 - d_{i}e^{i \theta})^{-1},
\] 
and consider $\phi(z)- \lambda$ as our new symbol. If we order the $a_{i}(\lambda)$ and $b_{i}(\lambda)$ as
\[ |a_{1}(\lambda) \leq |a_{2}(\lambda)| \leq \cdots \leq |a_{k}(\lambda)| \leq 1/|b_{1}(\lambda)| \leq 1/|b_{2}(\lambda)| \leq \cdots \leq 1/|b_{p -k}(\lambda)|,\]
then the determinant cannot vanish unless
\[  |a_{k}(\lambda)||b_{1}(\lambda)| = 1.\]

The Toeplitz plus Hankel case is more complicated. Recall that for $k=1$, 
\[ \begin{split} 
     \det(T_{n}(\phi) + H_{n}(\phi))  = &\frac{(1-b)(1+d)(1-cb)(1-ad)}{(1-bd)(1-ab)(1-cd)} \\
                                                    & + b^{n}d^{n}\frac{(1-ad)(b-d)(d-c)}{(d-a)(1-bd)(1-cd)} \\
                                                    & + b^{n}a^{n}\frac{(1-a)(a-c)(1+d)(b-d)}{(a-d)(1-ab)(1-cd)}.
     \end{split}
\]

Generically, the first term is nonzero, and then 
the above determinant will not vanish for large $n$ if $$|a(\lambda)||b(\lambda)| < 1 \,\,\,\mbox{and} \,\,\,  |d||b(\lambda)| < 1.$$

 However, if 
\[|a(\lambda)||b(\lambda)| = 1 \,\,\,\,\,\,\,\mbox{or} \,\,\,\,\,\,|d||b(\lambda)| = 1,\] 
 then the determinants could vanish. 

For general $k$ let 
\[ S = \{a_{i}(\lambda), \,\,1/b_{i}(\lambda), \,\,d_{i} \,\,| \,\,i = 1, \cdots , k \} \] and relabel those elements as $s_{i}(\lambda),\,\,i = 1, \cdots, 3k.$ If we arrange those elements in increasing order, then the determinant does not vanish unless
\[ |s_{2k}(\lambda)| = |s_{2k+1}(\lambda)| .\]

Following are eigenvalue plots for some Toeplitz plus Hankel matrices. The computational evidence (done with Mathematica) suggests that the limiting eigenvalues lie on the same curves as for the Toeplitz matrices. We do not have any examples to indicate otherwise, but we have not been able to rule out the possibility that the values of $|d_{i}|$ do  play a role. 

\begin{example}
We let
$$k = 1, \,\, n = 30,\,\,a = 1/5, \,\,b = i/2, \,\,c = 1/3, \,\,d = 1/4.$$ The curve is the image of the function. The red triangles are the T+H case and the blue circles, the Toeplitz case.

\vspace{.2in}
\centerline {
\includegraphics[width=3in]{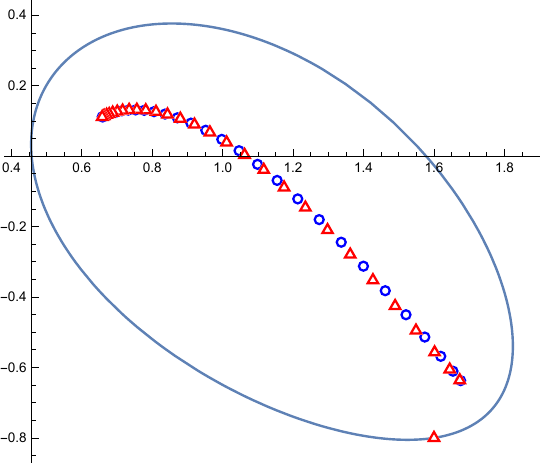}
}

\vspace{.2in}

\end{example}

\begin{example}
We let
$$k = 2, \,\, n = 50, \,\,a_{1} = 1/5, \,\,a_{2} = 3/5,\,\, b_{1} = i/2, \,\,c_{1} = 1/3, \,\,c_{2} = i/3, \,\,d_{1} = 1/4$$ 

\vspace{.2in}
\centerline {
\includegraphics[width=3in]{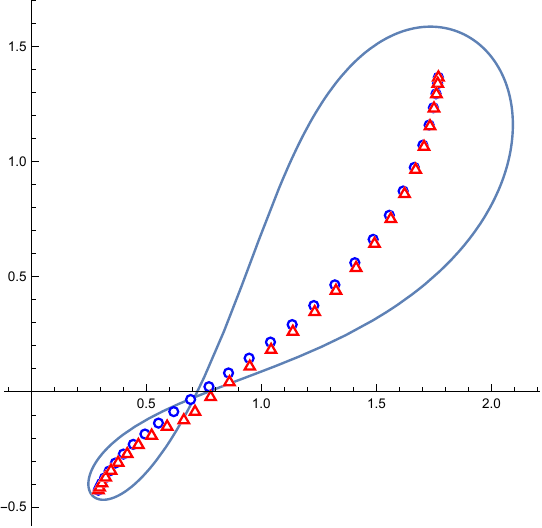}
}
\vspace{.2in}

\end{example}
Notice on the first plot there is an eigenvalue for the T+H case (and not the Toeplitz case) that lies on the curve itself. This appears to be at $\phi(1)$ and is present on many of our other numerical examples that are not contained here. 

To understand how this might occur let us go back to the case of $k = 1.$ The constant term in (\ref{det_formula_k=1}) has a factor of $1 -b.$ Letting $\lambda = \phi(1),$  the polynomial $$(z-a)(z-b) - \phi(1)(z-c)(z-d)$$ has a root at $z = 1.$ If this is the larger root in absolute value then $z = 1$ corresponds to $b = 1$ and thus the constant term vanishes and the other two terms in (\ref{det_formula_k=1}) tend to zero. 

\section{Appendices}
In the proof of the main theorem, two determinant identities were needed and in addition, an identity involving the Fourier coefficients of the function $\psi.$
One of the determinant identities was the standard one for a Cauchy determinant. The other for a Cauchy-type determinant we will prove here.
\subsection*{Appendix A: A Cauchy-type determinant identity}
\quad\\

If the $n \times n$ matrix $A$ has entries\[a_{i,j} = \frac{1}{(s_{i}-t_{j})(1-s_{i}t_{j})},\]  then
  \[  \det A = \frac{\prod_{i < j}(t_{i}-t_{j})(s_{j}-s_{i})(1- t_{i}t_{j})(1-s_{i}s_{j})}{\prod_{i,j}(s_{i}-t_{j})(1-t_{i}s_{j})}.\] 
 \proof
Note that if $s_{i} = s_{j}$ for any $i \ne j$, then two columns of the matrix are equal and the determinant is zero. The same holds if $t_{i} = t_{j}$ or $s_{i} = 1/s_{j}$ or  $t_{i} = 1/t_{j}$ for any $i \ne j.$
 
 Now think of the expanded determinant as rational function in the variables $s_{i}$ and $t_{i}.$ The denominator of the expanded determinant is $$\prod_{i,j}(s_{i}-t_{j})(1-t_{i}s_{j}).$$ By the above remarks, the numerator has factors of $$\prod_{i < j}(t_{i}-t_{j})(s_{j}-s_{i})(1- t_{i}t_{j})(1-s_{i}s_{j}).$$ 
 
 The degree of the denominator of the expanded determinant is $3n^{2}.$ The degree of the numerator is $3n^{2} - 3n = 3n(n-1).$ The degree of the expression $\prod_{i < j}(t_{i}-t_{j})(s_{j}-s_{i})(1- t_{i}t_{j})(1-s_{i}s_{j})$ is also $3n(n-1).$ Therefore we have
 \[  \det A = K\, \frac{\prod_{i < j}(t_{i}-t_{j})(s_{j}-s_{i})(1- t_{i}t_{j})(1-s_{i}s_{j})}{\prod_{i,j}(s_{i}-t_{j})(1-t_{i}s_{j})},\] for some constant $K.$ What is left to prove is that $K = 1.$ 
 
To see this, notice that if one expands the product $$\prod_{i < j}(t_{i}-t_{j})(s_{j}-s_{i})(1- t_{i}t_{j})(1-s_{i}s_{j}),$$ the terms of lowest orders come from choosing the ``$1$'' from all the factors of the form $(1- t_{i}t_{j})$ or $(1-s_{i}s_{j}).$ What is left is 
$\prod_{i < j}(t_{i}-t_{j})(s_{j}-s_{i})$ which is exactly the answer for a standard Cauchy determinant. But this is also true if we expand the determinant of $A$ and choose ``$1$'' in all the same factors in the numerator. Thus $K = 1.$
 \endproof  
 
\subsection*{Appendix B: Coefficient identity}
\quad\\

Recall that 
 \begin{align*}
 \al_{d^{+}_{j}} &=\prod_{i=1}^{k}\frac{(1-c_{i}d_{j})(1-b_{i}d_{j})(d_{j}-b_{i})}{(1- a_{i}d_{j})(1-d_{i}d_{j})}\prod_{i\neq j}\frac{1}{d_{j}-d_{i}} \\
 \al_{d^{-}_{j}} &= \prod_{i=1}^{k}\frac{(1-b_{i}d_{j})(d_{j}-c_{i})(d_{j}-b_{i})}{(d_{j}-a_{i})(1-d_{j}d_{i})}\prod_{i\neq j}\frac{1}{d_{j}-d_{i}} \\
 \al_{a_{j}} &= \prod_{i=1}^{k}\frac{(1-a_{j}b_{i})(a_{j}-c_{i})(a_{j}-b_{i})}{(1-d_{i}a_{j})(a_{j}-d_{i})}\prod_{i\neq j}\frac{1}{a_{j}-a_{i}}
 \end{align*}
 
Referring  to (\ref{vanishing-term}), we need to prove that 
 \[b_{i}\sum_{j=1}^{k}\left(- \al_{d^{-}_{j}}\frac{1}{d_{j}(d_{j}-b_{i})} - \al_{a_{j}}\frac{1}{a_{j}(a_{j}-b_{i})}   + \al_{d^{+}_{j}}\frac{1}{1- b_{i}d_{j}}\right) +   \prod_{j=1}^{k}\frac{b_{j}c_{j}}{a_{j}d_{j}} =0.\]
 
 We can regard 
\begin{equation}\label{monster} b_{i}\sum_{j=1}^{k}\left(- \al_{d^{-}_{j}}\frac{1}{d_{j}(d_{j}-b_{i})} - \al_{a_{j}}\frac{1}{a_{j}(a_{j}-b_{i})}   + \al_{d^{+}_{j}}\frac{1}{1- b_{i}d_{j}}\right)
 \end{equation}
 as a rational expression in the variables $d_{1}, \dots ,d_{k}.$
 Without loss of generality we assume $i=1$. Rewrite  (\ref{monster}) with a common denominator. Note that there are $3k$ terms in the numerator that are summed. The idea is to show that every factor of the denominator occurs in the numerator. 
  
 We start with the term $1-a_{1}d_{1}.$ Notice that in (\ref{monster}) only the two terms with factors of $\al_{a_{1}}$ and $\al^{+}_{d_{1}}$ have $1-a_{1}d_{1}$ in their denominators. Every other term will have $1-a_{1}d_{1}$ as a factor.
 
 The two remaining terms in the numerator will be, except for a common factor,
 \[\prod_{i=1}^{k}(1-c_{i}d_{1})(d_{1}-b_{i})\prod_{i\neq 1}(1- d_{1}b_{i})\prod_{i\neq 1}(a_{1}-a_{i})\prod_{i=1}^{k}(a_{1}-d_{i})\prod_{i\neq 1}(1-d_{i}a_{1})
 \]
 and
 \[d_{1}\prod_{i=1}^{k}(1-a_{1}b_{i})(a_{1}-c_{i})\prod_{i\neq 1}(d_{1}-d_{i})\prod_{i\neq 1}(a_{1}-b_{i})\prod_{i=1}^{k}(1-d_{1}d_{i}) \prod_{i\neq 1}(1-a_{i}d_{1}).
 \]
 If we let $d_{1} = 1/a_{1},$ then it is straight-forward to see that these are the same. Since we are subtracting these two terms in our expression we have a factor of $(1-d_{1}a_{1})$ in the numerator. 
This same proof works for any factor of the form $1-a_{i}d_{j}.$ 

Now we turn to a factor of the form $d_{j}-a_{i}.$ We may assume that $j=1$ and $i = 1$ and consider the terms involving $\al_{d_{1}^{-}}$ and $\al_{a_{1}}$. Using the same argument we have two terms in the numerator that do not have  $d_{j}-a_{i}$ as a factor. These, except for the common factors, are
\[ d_{1}\prod_{i=1}^{k}(1-a_{1}b_{i})(a_{1}-c_{i})\prod_{i\neq 1}(d_{1}-d_{i})\prod_{i\neq 1}(a_{1}-b_{i})\prod_{i=1}^{k}(1-d_{1}d_{i}) \prod_{i\neq 1}(d_{1}-a_{i})\]
and
\[a_{1}\prod_{i=1}^{k}(1-b_{i}d_{1})(d_{1}-c_{i})\prod_{i\neq 1}(d_{1}-b_{i})\prod_{i\neq 1}(a_{1}-a_{i})\prod_{i\neq1}^{k}(a_{1}-d_{i})\prod_{i\neq 1}(1-d_{i}a_{1}). 
 \]
Now let $d_{1} = a_{1}$ and we immediately see that these are the same. Since $\al_{d_{1}^{-}}$ has $d_{1}-a_{1}$ in its denominator and $\al_{a_{1}}$ has $a_{1}- d_{1}$ in its denominator the terms in the numerator cancel. Thus this shows that the terms in the summand have a factor of $d_{1}-a_{1}.$

The next factor that we consider is the type $1- d_{i}d_{j}.$ First we let $i=1,2$ and $j=1,2$ with $i \neq j.$ We have four terms in (\ref{monster}) with this factor in the denominator. Using the same approach as above, we consider the terms in the numerator involving $d^{-}_{1}$ and $d^{-}_{2}.$ These, except for common factors, are

\begin{align*}e_{1} =  \,\,\,& d_{2}\prod_{i=1}^{k}(1-b_{i}d_{1})(d_{1}-c_{i})\prod_{i\neq 1}(d_{1}-b_{i})\prod_{i=1}^{k}(1-a_{i}d_{1})\\&\times\prod_{i=1}^{k}(d_{2}-a_{i})\prod_{i > 2}(d_{2} -d_{i})\prod_{i  >1 }(1-d_{2}d_{i})\prod_{i =1}^{k}(1-a_{i}d_{2}),
\end{align*}
  \begin{align*}e_{2} = \,\,\,&d_{1}\prod_{i=1}^{k}(1-b_{i}d_{2})(d_{2}-c_{i})\prod_{i\neq 1}(d_{2}-b_{i})\prod_{i=1}^{k}(1-a_{i}d_{1})\\
 &\times \prod_{i=1}^{k}(d_{1}-a_{i})\prod_{i \neq 1}(d_{1} -d_{i})\prod_{i \neq 2 }(1-d_{1}d_{i})\prod_{i =1}^{k}(1-a_{i}d_{2}).
 \end{align*}
 For the terms involving  $d^{+}_{1}$ and $d^{+}_{2},$ we have
  \begin{align*} e_{3} = \,\,\,d_{1}d_{2}&\prod_{i=1}^{k}(1-c_{i}d_{1})(d_{1}-b_{i})\prod_{i\neq 1}(1- d_{1}b_{i})\prod_{i=1}^{k}(d_{1} -a_{i})\\
  \times &\prod_{i=1}^{k}(d_{2} - a_{i})\prod_{i=1}^{k}( 1- a_{i}d_{2})\prod_{i > 2 }(1-d_{2}d_{i})\prod_{i > 2}(d_{2} -d_{i}),
 \end{align*}
  \begin{align*}e_{4} = \,\,\,d_{1}d_{2}&\prod_{i=1}^{k}(1-c_{i}d_{2})(d_{2}-b_{i})\prod_{i\neq 1}(1- d_{2}b_{i})\prod_{i=1}^{k}(d_{1} -a_{i})\\
  \times &\prod_{i=1}^{k}(d_{2} - a_{i})\prod_{i=1}^{k}( 1- a_{i}d_{1})\prod_{i \neq 2 }(1-d_{1}d_{i})\prod_{i > 2}(d_{1} -d_{i}).\end{align*}
  
 We consider $e_{1}$ and $e_{4}$
and we replace $d_{1}$ with $1/d_{2}.$ Then these terms are the same and hence cancel in the numerator. 
The same argument works for the $e_{2}$ and $e_{3}$ terms above. 
 
 Now consider a term of the form $1-d_{1}^{2} = (1-d_{1})(1 + d_{1}).$
 The terms we need to consider arise from $\al_{1}^{+}$ and $\al_{1}^{-}.$ They are
\[\prod_{i=1}^{k}(1-b_{i}d_{1})(d_{1}-c_{i})\prod_{i\neq 1}(d_{1}-b_{i})\prod_{i=1}^{k}(1-a_{i}d_{1})\]
and
\[d_{1}\prod_{i=1}^{k}(1-c_{i}d_{1})(d_{1}-b_{i})\prod_{i\neq 1}(1- d_{1}b_{i})\prod_{i=1}^{k}(d_{1} -a_{i}).
\]
It is clear that if $d_{1} = \pm 1$ these two terms are the same and hence cancel in the summand.
Our final term to consider is one of the form $d_{1} - d_{2}$ The cancellation of these terms follows almost exactly as above and we leave it as an exercise for the reader. 

Now we once again think of our summand with all terms over a common denominator. All the terms involving $d_{j}$s have cancelled from the denominators except for the single factors of $d_{j}.$ It is not hard to see that the resulting rational function in the $d_{j}$s has a numerator with degree one less than the denominator in each of the $d_{j}$s. Hence we have a rational function of the form $$g(a_{i}, b_{i}, c_{i})/(d_{1}\cdots d_{k}).$$ It is also the case that the function $g$ must be the coefficient of the highest degree term in the polynomial of the numerator thought of as a function of $d_{1}$ divided by the coefficient of the highest degree term in the denominator. This value is $-\prod_{i=1}^{k}\frac{c_{i}b_{i}}{a_{i}}$ and that proves the statement.

While the reader may find the above argument a bit hard to follow, they may be reassured by the fact that the authors computed the $k=2$ case in Mathematica and found the computation of (\ref{monster}) to be correct. 

\subsection*{Appendix C: Mathematica Code for the determinant formula}

\begin{verbatim}
Z[a_, b_] := Product[1/(1 - a[[i]] b[[j]]), 
                     {i, 1, Length[a]}, {j, 1, Length[b]}]
ZO[a_] :=  Product[1/(1 - a[[i]] a[[j]]), 
                   {i, 1, Length[a] - 1}, {j, i + 1,  Length[a]}]
ZS[a_] := Product[1/(1 - a[[i]] a[[j]]), 
                  {i, 1, Length[a]}, {j, i, Length[a]}]
W[a_] := Product[1 - a[[i]], {i, 1, Length[a]}]
Y[a_] := Product[a[[i]], {i, 1, Length[a]}]

THDet[a_, b_, c_, d_, n_] :=   
    Module[{ad = Union[a, d], ssad, ssb},
    ssad[r_] := Subsets[ad, {r}];  
    ssb[r_] := Subsets[b, {r}];  
    W[-d]*Z[c, d]*ZS[d]*
    Sum[(-1)^r*Y[s]^(n - 1)*Y[t]^n*W[s]*W[Complement[b, t]]*
     Z[Union[Complement[ad, s], 1/t],Union[Complement[b, t], 1/s]]/
     (Z[Union[Complement[ad, s], 1/t], d]*
     Z[Union[Complement[b, t], 1/s], c]*
     ZO[Union[Complement[b, t], 1/s]]), 
     {r, 0, Min[Length[ad], Length[b]]}, {s, ssad[r]}, {t, ssb[r]}]]
\end{verbatim}

 \end{document}